\documentclass{article}
\usepackage[utf8]{inputenc}
\usepackage[paperwidth=7in, paperheight=10in, margin=.875in]{geometry}
 \usepackage[backref,colorlinks,linkcolor=red,anchorcolor=green,citecolor=blue]{hyperref}
\usepackage{amsfonts,amssymb}
\usepackage{amsmath}
\usepackage{amsthm}
\usepackage{graphicx}
\usepackage{enumerate}
\usepackage{verbatim}
\usepackage{enumerate}
\usepackage[utf8]{inputenc}
\usepackage{amsthm}
\usepackage{lipsum}
\usepackage{calligra}
\usepackage{calrsfs}
\usepackage{mathrsfs}
\usepackage{epstopdf}
\usepackage{subfigure}
\usepackage{dsfont}
\usepackage{tikz}
\usepackage{filecontents}
\usepackage{verbatim}
\usepackage{amssymb}
\usepackage[T1]{fontenc}
\usepackage[numbers]{natbib}

\newtheorem{theorem}{Theorem}
\newtheorem{lemma}[theorem]{Lemma}
\newtheorem{corollary}[theorem]{Corollary}
\newtheorem{proposition}[theorem]{Proposition}

\newtheorem{remark}[theorem]{Remark}

\newcommand*{\id}{{\normalfont\hbox{1\kern-0.15em \vrule width .8pt depth-.5pt}}}

\newcommand{\Tau}{\mathcal{T}}

\numberwithin{equation}{section}

\theoremstyle{definition}

{\swapnumbers}

\newcommand{\norm}[1]{\ensuremath{\left\|#1\right\|}}
\newcommand{\pa}[1]{\ensuremath{\left(#1\right)}}

\title{reaction periodic}
\author{Agustin Besteiro}
\date{June 2018}

\begin{document}
 \title{Existence of Peregrine Solitons in fractional reaction-diffusion equations\thanks{Received date, and accepted date (The correct dates will be entered by the editor).}}


          \author{Agust\'in Besteiro\thanks{Instituto de Matem\'atica Luis Santal\'o, CONICET--UBA, , Ciudad Universitaria, Pabell\'on I (1428) Buenos Aires, Argentina, (abesteiro@dm.uba.ar).}
          \and Diego Rial \thanks{Dpto. de Matem\'atica, FCEyN--UBA, Instituto de Matem\'atica Luis Santal\'o, CONICET--UBA,
          	Universidad de Buenos Aires, Ciudad Universitaria, Pabell\'on I (1428) Buenos Aires, Argentina, (drial@dm.uba.ar).}}

         \pagestyle{myheadings} \markboth{Existence of Peregrine Solitons in fractional reaction-diffusion equations}{A. Besteiro and D. Rial} \maketitle

          \begin{abstract}
       In this article, we will analyze the existence of Peregrine type solutions for the fractional diffusion reaction equation by applying Splitting-type methods. These functions that have two main characteristics, they are direct sum of functions of periodic type and functions that tend to zero at infinity. Global existence results are obtained for each particular characteristic, for then finally combining both results.
       
        \end{abstract}
          
          \providecommand{\keywords}[1]{\textbf{\textit{Keywords---}} #1}
          

\providecommand{\AMS}[1]{\textbf{\textit{AMS code---}} #1}
\section{Introduction}

We consider the non autonomous system
\begin{equation}
\label{eq: reaction-diffusion}
\partial_{t}u + \sigma(-\Delta)^{\beta} u = F(t,u),
\end{equation}
where $u(x,t)\in Z$  for $x\in\mathbb{R}^{n}$, $t>0$, $\sigma \ge 0$ and $0 < \beta \le 1$, $F:\mathbb{R} \times Z \to Z$ a continuous map and  $Z$ a Banach space.
We consider the initial problem $u(x,0)=u_{0}(x)$.

The aim of this paper is to prove the existence of Peregrine type of solutions for the fractional reaction diffusion equation, using recent numerical splitting techniques (\cite{Borgna2015}, \cite{DeLeo2015}, and \cite{Besteiro2018}) introduced for other purposes. 
Peregrine solitons were studied in (\cite{Peregrine1983}), and has multiple applications (See for example, \cite{Bailung2011}, \cite{Chabchoub2011}, \cite{Kibler2010}, \cite{Hammani2011} and \cite{Shrira2009}). "Peregrine solitons" are functions with two main characteristics: These are direct sum of periodic functions and functions that tend to zero when the spatial variable tends to infinity.

Fractional reaction-diffusion equations are frequently used on many different topics of applied mathematics such as biological models, population dynamics models, 
nuclear reactor models, just to name a few (see \cite{Baeumer2007}, \cite{bueno2014}, \cite{burrage2012} and references therein). 

The fractional model captures the faster spreading rates and power law invasion profiles observed in many applications compared to the classical model ($\beta=1$).
The main reason for this behavior is given by the fractional Laplacian, that is described by standard theories of fractional calculus (for a complete survey see \cite{Machado2011}). There are many different equivalent definitions 
of the fractional Laplacian and its behavior is well understood (see \cite{Bucur2016}, \cite{DiNezza2012}, \cite{Kwasnicki2017}, 
\cite{Lischke2018}, \cite{Silvestre2005}, \cite{Pozrikidis2016} and \cite{Landkof1972}). 

The non-autonomous nonlinear reaction diffusion equation dynamics were studied by \cite{Robinson2007} and others, 
analyzing the stability and evolution of the problem.

The paper is organized as follows: In Section 2 we set notations and preliminary results and in Section 3 we present main results, focusing first on each characteristic of the direct sum separately, for finally joining both results to reach the existence of Peregrine Solitons. 


\section{Notations and Preliminaries.}
 We are interested in continuous functions to vectorial values, that is to say, whose evaluations take values in Banach Spaces.

Let $Z$ be a Banach space, we define $C_{\rm u}(\mathbb{R}^{d},Z)$ as
the set of uniformly continuous and bounded functions on $\mathbb{R}^{d}$ with values in $Z$. Taking the norm
\begin{equation*}
\|u\|_{\infty,Z} = \sup_{x\in\mathbb{R}^{d}}|u(x)|_{Z},
\end{equation*}
$C_{\rm u}(\mathbb{R}^{d},Z)$ is a Banach space. 
It is easy to see that if $g \in L^1(\mathbb{R}^d)$ and $u \in C_u(\mathbb{R}^d,Z)$ the Bochner integral is defined in the following way,
\begin{align*}
\pa{g*u}(x) = \int_{\mathbb{R}^{d}} g(y)u(x-y) dy
\end{align*}
This defines an element of $C_{\rm u}(\mathbb{R}^{d},Z)$ and the linear operator $u\mapsto g*u$ is continuous (see \cite{Cazenave1998}).

The following results show that the operator $-(-\Delta)^{\beta}$ defines a continuous contraction semigroup 
in the Banach space $C_{\rm u}(\mathbb{R}^{d},Z)$.
The following lemma is a consequence of L\'evy--Khintchine formula for infinitely divisible distributions and 
the properties of the Fourier transform.

\begin{lemma}
Let $0<\beta \le 1$ and $g_{\beta}\in C_{0}(\mathbb{R}^{d})$ such that $\hat{g}_{\beta}(\xi) = e^{-|\xi|^{2\beta}}$,
it holds $g_{\beta}$ is positive, invariant under rotations of $\mathbb{R}^{d}$, integrable and
\begin{align*}
\int_{\mathbb{R}^{d}}g_{\beta}(x)dx = 1.
\end{align*}
\end{lemma}

\begin{proof}
The first statement follows from Theorem 14.14 of \cite{Sato1999},
the remaining claims are immediate from the definition of $\hat{g}_{\beta}$.
\end{proof}

Based on the previous lemma, we study Green's function associated to the linear operator $\partial_{t} + \sigma(-\Delta)^{\beta}$.

\begin{proposition}
Let $\sigma > 0$ and $0<\beta \le 1$, the function $G_{\sigma,\beta}$ given by 
\begin{align*}
G_{\sigma,\beta}(t,x) = (\sigma t)^{-\frac{n}{2\beta}} g_{\beta}((\sigma t)^{-\frac{1}{2\beta}} x),
\end{align*}
verifies
\begin{enumerate}[i.]
\item $G_{\sigma,\beta}(.,t)>0$;
\item $G_{\sigma,\beta}(.,t)\in L^{1}(\mathbb{R}^{d})$ and 
\begin{align*}
\int_{\mathbb{R}^{d}}G_{\sigma,\beta}(t,x)dx = 1;
\end{align*}
\item \label{it: semi} $G_{\sigma,\beta}(\cdot,t)*G_{\sigma,\beta}(\cdot,t') = G_{\sigma,\beta}(\cdot,t+t')$, for $t,t'>0$;
\item $\partial_{t}G_{\sigma,\beta} + \sigma (-\Delta)^{\beta} G_{\sigma,\beta} = 0$ for $t>0$.
\end{enumerate}
\end{proposition}
\begin{proof}
The first and second statements are a consequence of the definition of $\hat{g}_{\beta}$.
The third and fourth statements are immediate applying Fourier transform. 
\end{proof}

In the following proposition, we have that the linear operator $-\sigma(-\Delta)^{\beta}$ defines a contraction continuous semigroup in the set $C_{\rm u}(\mathbb{R}^{d},Z)$.

\begin{proposition}
For any $\sigma > 0$ and $0<\beta \le 1$, the map ${\sf S}:\mathbb{R}_{+}\to \mathcal{B}(C_{\rm u}(\mathbb{R}^{d},Z))$ defined by
${\sf S}(t)u = G_{\sigma,\beta}(.,t)*u$ is a continuous contraction semigroup.
\end{proposition}
\begin{proof}
The proof can be found in \cite{Besteiro2018} Proposition 2.2.
\end{proof}
\begin{remark}
If $u\in C_{\rm u}(\mathbb{R}^{d},Z)$ is a constant, then ${\sf S}(t)u=u$.
\end{remark}
 In this paper, we consider integral solutions of the problem \eqref{eq: reaction-diffusion}.
We say that $u\in C([0,T],C_{\rm u}(\mathbb{R}^{d},Z))$ is a mild solution of \eqref{eq: reaction-diffusion}
iff $u$ verifies
\begin{align}
\label{eq: mild solution}
u(t) = {\sf S}(t)u_{0} + \int_{0}^{t} {\sf S}(t-t')F(t',u(t')) dt'.
\end{align}
Since our method to build solutions of \eqref{eq: mild solution} is based on the application of the Lie-Trotter method, it is necessary to study the non-linear problem associated with $ F $. 
We remark that some regularity condition is necessary for convergence, as it is shown in the counterexample given in \cite{Canzi2012}.

Let $F:\mathbb{R}_{+} \times Z \to Z$ be a continuous map, we say that is locally Lipschitz 
in the second variable if, given $R,T>0$ there exists $L = L(R,T)>0$ such that if $t\in [0,T]$ and $z,\tilde{z}\in Z$ 
with $|z|_{Z},|\tilde{z}|_{Z} \le R$, then
\begin{align*}
|F(t,z) - F(t,\tilde{z})|_{Z} \le L |z - \tilde{z}|_{Z}.
\end{align*}
In this case, for any $z_{0}\in Z$
there exists a unique (maximal) solution of the Cauchy problem
\begin{align}
\label{eq: integral equation}
z(t) = z_{0} + \int_{t_{0}}^{t}F(t',z(t'))dt'
\end{align}
defined on $[t_{0},t_{0}+T^{*}(t_{0},z_{0}))$, with $T^{*}(t_{0},z_{0})$ is the maximal time of existence.
It is easy to see that 
there exists a nonincreasing function $\mathcal{T}:\mathbb{R}_{+}^{2} \to \mathbb{R}_{+}$,
such that
\begin{align*}
 \mathcal{T}(T,R)\le \inf\{T^{*}(t_{0},z_{0}):0\le t_{0} \le T, |z_{0}|_{Z} \le R\}.
\end{align*}
Also, one of the following alternatives holds:
\begin{itemize}
\item[-] $T^{*}(t_{0},z_{0}) = \infty$;
\item[-] $T^{*}(t_{0},z_{0}) < \infty$ and $|z(t)|_{Z} \to \infty$ when $t \uparrow t_{0} + T^{*}(t_{0},z_{0})$.
\end{itemize}
We can see that $F:\mathbb{R}_{+}\times C_{\rm u}(\mathbb{R}^{d},Z)\to C_{\rm u}(\mathbb{R}^{d},Z)$,
given by $F(t,u)(x) = F(t,u(x))$ is continuous and locally Lipschitz in the second variable.
Therefore, we can consider problem \eqref{eq: integral equation} in $C_{\rm u}(\mathbb{R}^{d},Z)$.
%
We denote by $\mathsf{N}:\mathbb{R} \times \mathbb{R} \times C_{\rm u}(\mathbb{R}^{d},Z)\to C_{\rm u}(\mathbb{R}^{d},Z)$ the flow 
generated by the integral equation \eqref{eq: integral equation} as $u(t)=\mathsf{N}(t,t_0,u_0)$,
defined for $t_{0} \le t < t_{0} + T^{*}(t_{0},u_{0})$.

The following result relates the solutions of \eqref{eq: integral equation} with the problem \eqref{eq: mild solution} in the case of having constant initial data.

\begin{proposition}
\label{pr: constant}
If $u_{0}$ is a constant function, then $u(t) = \mathsf{N}(t,t_0,u_0)$ is a solution of \eqref{eq: mild solution}.
\end{proposition}
\begin{proof}
Since $u_{0}$ is a constant function, from the uniqueness of the problem \eqref{eq: integral equation}, we have
$u(t)$ is a constant function for any $t>0$ where the solution is defined. Therefore,
\begin{align*}
u(t) & = u_{0} + \int_{0}^{t} F(t',u(t')) dt'
= {\sf S}(t)u_{0} + \int_{0}^{t} {\sf S}(t-t')F(t',u(t')) dt',
\end{align*}
which proves our assertion.
\end{proof}

\begin{theorem}
	\label{th: local existence}
	There exists a  function $T^{*}:C_{\rm u}(\mathbb{R}^{d},Z)\to \mathbb{R}_{+}$ such that
	for $u_{0}\in C_{\rm u}(\mathbb{R}^{d},Z)$, exists a unique $u\in C([0,T^{*}(u_{0})),C_{\rm u}(\mathbb{R}^{d},Z))$ mild solution
	of \eqref{eq: reaction-diffusion} with $u(0) = u_{0}$. Moreover, one of the following alternatives holds:
	\begin{itemize}
		\item $T^{*}(u_{0}) = \infty$;
		\item $T^{*}(u_{0}) < \infty$ and $\lim_{t \uparrow T^{*}(u_{0})}\|u(t)\|_{C_{\rm u}(\mathbb{R}^{d},Z)} = \infty$.
	\end{itemize}
\end{theorem}
\begin{proof}
	See Theorem 4.3.4 in \cite{Cazenave1998}.
\end{proof}

\begin{proposition}
	\label{pr: continuous dependence}
	Under conditions of theorem above, then
	\begin{enumerate}
		\item $T^{*}:C_{\rm u}(\mathbb{R}^{d},Z)\to \mathbb{R}_{+}$ is lower semi-continuous;
		\item If $u_{0,n} \to u_{0}$ in $C_{\rm u}(\mathbb{R}^{d},Z)$ and $0 < T < T^{*}(u_{0})$, then 
		$u_{n} \to u$ in the Banach space $C([0,T],C_{\rm u}(\mathbb{R}^{d},Z))$.
	\end{enumerate}
\end{proposition}
\begin{proof}
	See Proposition 4.3.7 in \cite{Cazenave1998}.
\end{proof}

\section{Periodic solutions}
\label{Sec.3}

In this section, we will analyze the existence of solutions for the fractional reaction diffusion equation by applying Splitting methods to functions that have two main characteristics: these are direct sum of functions of periodic type and functions that tend to zero at infinity. This type of solution is also studied in the non-linear Schroedinger equation, under the name of "Peregrine solitons" \cite{Peregrine1983}. Well posedness results are obtained for each particular characteristic, to then combine both results.
In addition, we will observe that the evolution of the periodic part is independent of the part that tends to zero at infinity.

For instance, suppose that the non-linearity is of polynomial type (as in the Fitzhugh-Nagumo equation, see \cite{Asgari2011}), in this case we use $F(u)=u^2$.
If $u(t)=v(t)+w(t)$, where $v(t)$ is a periodic function and 
$w(t)$ is a function that tends to zero when the spatial variable tends to infinity, then we have that

\begin{align*}
F(u)=F(v+w)=(v+w)^2=v^2+2vw+w^2    
\end{align*}
where, $v^2$ is periodic and $2vw+w^2$ tends to zero.
In this specific case we can appreciate the \textit{absorption} of the part that tends to zero, in the crossed terms. As $v^2=F(v)$, we expect that the periodic part of the initial data evolve independently from the part that tends to zero for the non linear equation. In this section we obtain general results to which this example refers.

Let $\{\gamma_{1},\ldots,\gamma_{q}\}$ be $q$ linearly independent vectors of $\mathbb{R}^{d}$ and 
let $\Gamma$ be the lattice generated, i.e., $\Gamma=\{n_{1}\gamma_{1}+\cdots+n_{q}\gamma_{q}: n_{j}\in\mathbb{Z}\}$.
A function $u\in C_{\rm u}(\mathbb{R}^{d},Z)$ is $\Gamma$--periodic if $u(x+\gamma) =u(x) $
for any $\gamma\in\Gamma$.
We denote the set of $\Gamma$--periodic functions of $C_{\rm u}(\mathbb{R}^{d},Z)$ by $C_{\rm u}(\mathbb{R}^{d}/\Gamma,Z)$.

We consider the space $C_{0}(\mathbb{R}^{d},Z)$ of functions which converge to $0$ when $|x|\to\infty$. 
It is easy to prove the following result.
\begin{proposition}
$C_{\rm u}(\mathbb{R}^{d}/\Gamma,Z),C_{0}(\mathbb{R}^{d},Z)\subset C_{\rm u}(\mathbb{R}^{d},Z)$ are closed subspaces. 
Moreover, 
$C_{\rm u}(\mathbb{R}^{d}/\Gamma,Z)\cap C_{0}(\mathbb{R}^{d},Z)=\{0\}$.
\end{proposition}
\begin{proof}
 Let $u\in C_{\rm u}(\mathbb{R}^{d}/\Gamma,Z)$, we set $x\in\mathbb{R}^{d}$
 $u(x)= \lim_{|\gamma|\to\infty}u(x+\gamma)$. If $u\in C_{0}(\mathbb{R}^{d},Z)$, then
 $\lim_{|\gamma|\to\infty}u(x+\gamma)=0$. Therefore, $u(x)=0$ for any $x\in\mathbb{R}^{d}$.
\end{proof}
\begin{lemma}
Let $X$ be a Banach space and let $X_{1},X_{2}\subset X$ be closed subspaces
such that $ X_{1}\bigcap X_{2}=\{0\}$, the following statement are equivalent
\begin{enumerate}[i.]
	\item $X_{1}\oplus X_{2}$ is closed.
	\item The projector $P:X_{1}\oplus X_{2} \to X_{1}$ is continuous.
\end{enumerate}
\end{lemma}
\begin{proof}
Since $X_{1} \oplus X_{2}$ is a Banach space, the linear  map $\phi:X_{1}\times X_{2} \to X_{1} \oplus X_{2}$
given by $\phi(x_{1},x_{2}) = x_{1}+x_{2}$
is bijective, and from the closed graph theorem we have $\phi$ and $\phi^{-1}$ are continuous operator.
We can write $P=\pi_{1} \phi^{-1}$ and then $P$ is continuous.
On the other hand, $X_{1}\oplus X_{2} = P^{-1}X_{1}$, since $P$ continuous and $X_{1}$ a closed subspace,
$X_{1}\oplus X_{2}$ is closed.
\end{proof}
\begin{lemma}
\label{le: proyector}
 The projector $P:C_{\rm u}(\mathbb{R}^{d}/\Gamma,Z)\oplus C_{0}(\mathbb{R}^{d},Z) \to C_{\rm u}(\mathbb{R}^{d}/\Gamma,Z)$
 is continuous.
\end{lemma}
\begin{proof}
 Let $u=v+w\in C_{\rm u}(\mathbb{R}^{d}/\Gamma,Z)\oplus C_{0}(\mathbb{R}^{d},Z)$,
 $v\in C_{\rm u}(\mathbb{R}^{d}/\Gamma,Z)$ and $w\in C_{0}(\mathbb{R}^{d},Z)$. For any $x\in\mathbb{R}^{d}$, we can see that
 \begin{align*}
 v(x) = \mathop{\lim_{|\gamma|\to\infty}}_{\gamma\in\Gamma}v(x+\gamma) = 
 \mathop{\lim_{|\gamma|\to\infty}}_{\gamma\in\Gamma}u(x+\gamma),
 \end{align*}
 then $|v(x)|\le \|u\|_{C_{\rm u}(\mathbb{R}^{d},Z)}$, 
 which implies $\|v\|_{C_{\rm u}(\mathbb{R}^{d},Z)} = \|Pu\|_{C_{\rm u}(\mathbb{R}^{d},Z)} \le \|u\|_{C_{\rm u}(\mathbb{R}^{d},Z)}$.
\end{proof}
\begin{corollary}
The direct sum $X_{\Gamma,Z} = C_{\rm u}(\mathbb{R}^{d}/\Gamma,Z)\oplus C_{0}(\mathbb{R}^{d},Z)$
is a closed subspace of $C_u(\mathbb{R}^{d},Z)$.
\end{corollary}

To obtain the existence of solutions in the space $ X_{\Gamma, Z} $, we first study each case separately. We analyze the existence of solutions for the case of $ \Gamma $ periodic functions using the translation function.

Given $\gamma\in\mathbb{R}^{d}$ define ${\sf T}_{\gamma}: C_{\rm u}(\mathbb{R}^{d},Z) \to C_{\rm u}(\mathbb{R}^{d},Z)$ as $({\sf T}_{\gamma}u)(x) = u(x+\gamma)$.
Since ${\sf S}(t)$ is a convolution operator, it is easy to see that
${\sf T}_{\gamma}{\sf S}(t) = {\sf S}(t) {\sf T}_{\gamma}$.
Using that ${\sf T}_{\gamma}F(t,u) = F(t,{\sf T}_{\gamma}u)$ we obtain
\begin{align*}
{\sf T}_{\gamma}u(t) = {\sf S}(t){\sf T}_{\gamma}u_{0} + \int_{0}^{t}{\sf S}(t-t')F(t,{\sf T}_{\gamma}u(t'))dt'.
\end{align*}
Therefore, ${\sf T}_{\gamma}u$ is the solution of \eqref{eq: mild solution}
with initial data ${\sf T}_{\gamma}u_{0}$.
\begin{proposition}
\label{pr: globalexistenceperiodic}
 If $u_{0}\in C_{\rm u}(\mathbb{R}^{d}/\Gamma,Z)$, then the solution $u$ of the equation 
 \eqref{eq: mild solution} verifies $u(t)\in C_{\rm u}(\mathbb{R}^{d}/\Gamma,Z)$
for $0\le t <T^{*}(u_{0})$.
\end{proposition}
\begin{proof}
	Since ${\sf T}_{\gamma}u_{0} = u_{0}$ for any $\gamma\in\Gamma$,
	${\sf T}_{\gamma}u,u$ are solutions with the same initial data.
	From uniqueness, we have ${\sf T}_{\gamma}u = u$.
	Therefore, $u(t)\in C_{\rm u}(\mathbb{R}^{d}/\Gamma,Z)$.
\end{proof}

We now analyze the existence of solutions of functions that tend to zero when the spatial variable tends to infinity.

\begin{lemma}
\label{le: S(t)C0}
If $u \in  C_{0}(\mathbb{R}^{d},Z)$, then ${\sf S}(t)u \in  C_{ 0}(\mathbb{R}^{d},Z)$ for $t\in\mathbb{R}_{+}$.
\end{lemma}
\begin{proof}
Let $\{x_{n}\}_{n\in\mathbb{N}}$ be a sequence verifying $|x_{n}|\to\infty$, we have
\begin{align*}
|({\sf S}(t)u)(x_{n})|_{Z}\le \int_{\mathbb{R}^{d}}G_{\sigma,\beta}(t,y)|u(x_{n}-y)|_{Z} dy. 
\end{align*}
As $G_{\sigma,\beta}(t,.)|u(x_{n}-.)|_{Z}\le G_{\sigma,\beta}(t,.)\|u\|_{\infty,Z}$
and $G_{\sigma,\beta}(t,y)|u(x_{n}-y)|_{Z}\to 0$,
from dominated convergence theorem we obtain that
$\lim_{n\to\infty}|({\sf S}(t)u)(x_{n})|_{Z}=0$. Since $\{x_{n}\}_{n\in\mathbb{N}}$
an arbitrary sequence, we have that
${\sf S}(t)u\in C_{0}(\mathbb{R}^{d},Z)$.
\end{proof}
\begin{lemma}
\label{le: u-tilde u}
 Let $u_{0},\tilde{u}_{0} \in C_{\rm u}(\mathbb{R}^{d},Z)$,
 if $u_{0}-\tilde{u}_{0} \in C_{0}(\mathbb{R}^{d},Z)$, 
 then ${\sf N}(t,t_0,u_{0})-{\sf N}(t,t_0,\tilde{u}_{0})\in C_{0}(\mathbb{R}^{d},Z)$
 for $0\le t <\min\{T^{*}(u_{0}),T^{*}(\tilde{u}_{0})\}$.
\end{lemma}
\begin{proof}
Let $u(t) = {\sf N}(t,t_0,u_{0})$ and $\tilde{u}(t)={\sf N}(t,t_0,\tilde{u}_{0})$, for any
$x\in\mathbb{R}^{d}$ we have
\begin{align*}
|u(x,t) -\tilde{u}(x,t)|_{Z} & \le |u_{0}(x) - \tilde{u}_{0}(x)|_{Z} 
+ \int_{0}^{t} |F(t',u(x,t')) - F(t',\tilde{u}(x,t'))|_{Z} dt' \\
& \le |u_{0}(x) - \tilde{u}_{0}(x)|_{Z} 
+ L \int_{0}^{t} |u(x,t') - \tilde{u}(x,t')|_{Z} dt'.
\end{align*}
From Gronwall's lemma, we get
$|u(x,t) -\tilde{u}(x,t)|_{Z} \le e^{Lt}|u_{0}(x) - \tilde{u}_{0}(x)|_{Z} $.
Given $\varepsilon>0$, there exists $r>0$ such that
$|u_{0}(x) - \tilde{u}_{0}(x)|_{Z}<\varepsilon e^{-Lt}$ for $|x|>r$, then
$|u(x,t) -\tilde{u}(x,t)|_{Z} < \varepsilon$, which implies 
$u(t) -\tilde{u}(t)\in C_{0}(\mathbb{R}^{d},Z)$.
\end{proof}
\begin{proposition}
\label{pr: localsolperiodic}
 Let $u_{0},\tilde{u}_{0} \in C_{\rm u}(\mathbb{R}^{d},Z)$,
 such that $u_{0}-\tilde{u}_{0} \in C_{0}(\mathbb{R}^{d},Z)$ and let $u,\tilde{u}$ be 
 the respective solutions of \eqref{eq: mild solution}. For any $0\le t <\min\{T^{*}(u_{0}),T^{*}(\tilde{u}_{0})\}$,
 it is verified $u(t)-\tilde{u}(t)\in C_{0}(\mathbb{R}^{d},Z)$.
\end{proposition}
\begin{proof}
 For $t \in [0,\min\{T^{*}(u_{0}),T^{*}(\tilde{u}_{0})\})$, let $n\in\mathbb{N}$, $h=t/n$ and $\{U_{h,k}\}_{0\le j\le n}$,$\{\tilde{U}_{h,k}\}_{0\le j\le n}$ 
 sequences defined in terms of a recurrence, in the following way:
 
 Let $\{U_{h,k}\}_{0\le k\le n},\{V_{h,k}\}_{1\le k\le n}$ be the sequences given by $U_{h,0} = u_{0}$,
\begin{subequations}
\label{eq: L-T}
\begin{align}
V_{h,k+1} & = {\sf S}(h) U_{h,k}, \\
U_{h,k+1} & = \mathsf{N}(kh+h,kh+h/2,V_{h,k+1}), \quad k=0,\dots,n-1.
\end{align}
\end{subequations}

 We claim that $U_{h,k}-\tilde{U}_{h,k}\in C_{0}(\mathbb{R}^{d},Z)$ for $k=0,\dots,n$.
 Clearly, the assertion is true for $k=0$.
 If $U_{h,k-1}-\tilde{U}_{h,k-1}\in C_{0}(\mathbb{R}^{d},Z)$, from Lemma \ref{le: u-tilde u},
 we have ${\sf N}(kh,kh-h/2,V_{h,k-1})-{\sf N}(kh,kh-h/2,\tilde{V}_{h,k-1})\in C_{0}(\mathbb{R}^{d},Z)$. Using lemma \ref{le: S(t)C0}, we can see that
\begin{align*}
 V_{h,k}-\tilde{V}_{h,k} 
= S(h)({\sf N}(kh,kh-h/2,V_{h,k-1})-{\sf N}(kh,kh-h/2,\tilde{V}_{h,k-1}))\in C_{0}(\mathbb{R}^{d},Z).
\end{align*}
As $C_{0}(\mathbb{R}^{d},Z)$ is closed and $U_{h,n}-\tilde{U}_{h,n}\to u(t)-\tilde{u}(t)$,
we obtain the result. 
\end{proof}

\begin{theorem}
 Let $u_{0}\in X_{\Gamma,Z}$, the solution $u$ of the equation
 \eqref{eq: mild solution} verifies $u(t)\in X_{\Gamma,Z}$ for $0\le t <T^{*}(u_{0})$.
Moreover, if $u_{0} = v_{0} + w_{0}$ with $v_{0}\in C_{\rm u}(\mathbb{R}^{d}/\Gamma,Z)$
and $w_{0} \in C_{0}(\mathbb{R}^{d},Z)$, then $u(t)=v(t)+w(t)$, where
$v$ is the solution of \eqref{eq: mild solution} with initial data $v_{0}$ and
$w$ is the solution of
\begin{align*}
 w(t) = {\sf S}(t)w_{0} + \int_{0}^{t} {\sf S}(t-t')\left(F(v(t') + w(t')) - F(v(t')) \right) dt'.
\end{align*}
\end{theorem}
\begin{proof}
As $u_{0}\in X_{\Gamma,Z} \subset C_{\rm u}(\mathbb{R}^{d},Z)$, by theorem \ref{th: local existence} we have that $u(t) \in C_{\rm u}(\mathbb{R}^{d},Z)$ with maximal time of existence $T^*(u_0)$.
We observe that as $v_0 \in C_{\rm u}(\mathbb{R}^{d}/\Gamma,Z)$ then by proposition \ref{pr: globalexistenceperiodic} we know that $v(t) \in C_{\rm u}(\mathbb{R}^{d}/\Gamma,Z)$ with maximal time of existence $T^*(v_0)$.
We define $w(t)=u(t)-v(t)$.
By hypothesis, we have that $w_{0}=w(0)=u(0)-v(0)=u_{0}-v_{0} \in C_{0}(\mathbb{R}^{d},Z) $ therefore, by proposition \ref{pr: localsolperiodic} we know that $w(t) \in C_{0}(\mathbb{R}^{d},Z)$. Then, we obtain that $u(t)=v(t)+w(t) \in X_{\Gamma,Z}$, where $v(t) \in C_{\rm u}(\mathbb{R}^{d}/\Gamma,Z)$
and $w(t) \in C_{0}(\mathbb{R}^{d},Z)$ in the interval $[0,T_{min})$ donde $T_{min}= \min \{T(u_0),T(v_0)\}$. If it were that $T^*(v_0) \geq T^*(u_0)$ then we have the result.

Suppose that $T^*(v_0)< T^*(u_0)$.

Let $T \in (0,T^*(u_0))$, $M=\max_{0\leq t \leq T} \norm{u(t)}$. We define $\Tau=\{ t \in [0,T]: u(t) \notin X_{\Gamma,Z} \}$, that is, the times for which we have that $u(t)$ is not a direct sum. Suppose that $\Tau \neq \varnothing$.
Then there exists $t_1 = \inf{\Tau}. $ 
We analize if the infimum can be equal to zero or greater than zero.

The case in which $t_1=0$ is not possible because we have already seen that $u(t) \in X_{\Gamma,Z}$, in the interval $[0,T^*(v_0))$.
In the same way, if $t_1>0$ and additionally $t_1 < T^*(v_0)$ we have that $u(t) \in X_{\Gamma,Z}$. We analize the remaining case, $t_1>0$ and $T>t_1 > T^*(v_0)$.

We observe that, by theorem \ref{th: local existence} we obtain that $\lim_{t \to T^*(v_0)}\norm{v(t)}=+ \infty$ but on the other hand, by lemma \ref{le: proyector} we have that $\norm{v(t)} \leq \norm{P} \norm{u(t)} \leq \norm{P} M$ that is, the norm $v(t)$ is bounded for $t \in [0,T^*(v_0)) \subset [0,T]$, which is a contradiction. 

 So we finally have that $u(t)\in X_{\Gamma,Z}$ for $t \in [0,T^*(u_0))$.


\end{proof}

  \section*{Acknowledgement}
This work was partially supported by CONICET--Argentina, PIP 11220130100006.

\end{document}